\documentclass[12pt]{article}
\usepackage{amsmath,amsfonts,amssymb,amsthm,amsbsy,euscript}
\usepackage{graphicx} 

\newtheorem{theor}{Theorem}
\newtheorem{prop}[theor]{Proposition}
\newtheorem{lemma}[theor]{Lemma}
\theoremstyle{definition}

\newtheorem{example}{Example}

\theoremstyle{remark}
\newtheorem{rem}{Remark}

\newcommand{\cev}[1]{\reflectbox{\ensuremath{\vec{\reflectbox{\ensuremath{#1}}}}}}

\newcommand{\pinner}{\mathbin{\mathchoice
{\hbox{\vrule width0.6em depth0pt height0.4pt
	\vrule width0.4pt depth0pt height0.8ex}}
{\hbox{\vrule width0.6em depth0pt height0.4pt
	\vrule width0.4pt depth0pt height0.8ex}}
{\hbox{\kern0.14em
	\vrule width0.48em depth0pt height0.4pt
	\vrule width0.4pt depth0pt height0.6ex\kern0.14em}}
{\hbox{\kern0.1em
	\vrule width0.39em depth0pt height0.4pt
	\vrule width0.4pt depth0pt height0.5ex\kern0.1em}}}}
\newcommand{\inner}{\pinner\,}

\newcommand{\schouten}[1]{[\![#1]\!]}
\newcommand{\Or}{{\rm O\oldvec{r}}}
\newcommand{\cP}{\mathcal{P}}
\newcommand{\cQ}{\mathcal{Q}}
\newcommand{\cR}{\mathcal{R}}

\newcommand{\cX}{{\EuScript X}}    
\newcommand{\cY}{{\EuScript Y}}    
\newcommand{\cZ}{{\EuScript Z}}    
\newcommand{\cV}{{\EuScript V}}    
\newcommand{\cE}{{\EuScript E}}    

\newcommand{\dd}{\partial}
\newcommand{\veps}{\varepsilon}
\newcommand{\Id}{\mathrm{d}}
\DeclareMathOperator{\Lie}{L}

\DeclareMathOperator{\tr}{tr}
\DeclareMathOperator{\dvol}{dvol}

\newcommand{\lshad}{[\![}
\newcommand{\rshad}{]\!]}

\def\oldvec{\mathaccent "017E\relax }

\newcommand{\by}[1]{\textrm{{#1}}}
\newcommand{\jour}[1]{\textrm{{#1}}}
\newcommand{\vol}[1]{\textrm{{V.~#1.}}}
\newcommand{\book}[1]{\textrm{{#1}}}

\newcommand{\coloneqq}{\mathrel{{:}{=}}}

\title{Universal cocycles and the 
graph complex action on homogeneous 
Poisson brackets by diffeomorphisms}

\author{R. Buring\thanks{%
Mathematical Institute, 
Johannes Gutenberg University of Mainz, 
Staudingerweg~9, 
\mbox{D-\/55128} 
Germany.\quad
A part of this research was done while R.B.\ was visiting at the IH\'ES
, supported by MI JGU project~5020.},%
$^{,\P}$
\quad 
A.\,V.\,Kiselev\thanks{%
Bernoulli Institute for Mathematics, Computer Science \& 
Artificial Intelligence, University of Groningen, P.O.~Box~407, 9700~AK Groningen, The~Netherlands.}
\,
\thanks{\textit{Current address}: 
Institut des Hautes \'Etudes Scientifiques (IH\'ES),
Le Bois\/--\/Marie, 35~route de Chartres, Bures\/-\/sur\/-\/Yvette, 91440 France.}
\,
\thanks{
Supported 
by BI~RUG project~135110 (Groningen) 
and the IH\'ES (in part, by the Nokia Fund).
}\,$^{\,,\pounds}$
}

\date{{E-mail}:\quad 
${}^{\P}$
\texttt{rburing\symbol{"40}uni-mainz.de},\quad
${}^{\pounds}$
\texttt{A.V.Kiselev\symbol{"40}rug.nl}%
}

\begin{document}
\maketitle
\begin{abstract}\noindent%
The graph complex acts on the spaces of Poisson bi\/-\/vectors $\cP$ by infinitesimal symmetries.
We prove that whenever a Poisson structure is homogeneous, i.e.\ $\cP=\Lie_{\vec{V}}(\cP)$
w.r.t.\ the Lie derivative along some vector field $\smash{\vec{V}}$, but not quadratic
(the coefficients of $\cP$ are not degree\/-\/two homogeneous polynomials),
and whenever its velocity bi\/-\/vector $\dot{\cP}=\cQ(\cP)$, 
also homogeneous w.r.t.\ $\smash{\vec{V}}$ by $\Lie_{\vec{V}}(\cQ)=n\cQ$
whenever $\cQ(\cP)=\Or(\gamma)(\cP^{\otimes^n})$ is obtained using the orientation morphism~$\Or$
from a graph cocycle~$\gamma$ on $n$ vertices and $2n-2$ edges, 
then the $1$\/-\/vector $\smash{\vec{\cX}}=\Or(\gamma)(\smash{\vec{V}}\otimes\smash{\cP^{\otimes^{n-1}}})$ is a Poisson cocycle.
Its 
construction is uniform 
for all Poisson bi\/-\/vectors $\cP$ satisfying the above assumptions, 
on all finite\/-\/dimensional affine manifolds~$M$.
Still, if the bi\/-\/vector $\cQ\not\equiv0$ is exact in the respective Poisson cohomology, so 
there exists a vector field $\smash{\vec{\cY}}$ such that $\cQ(\cP)=\lshad\smash{\vec{\cY}},\cP\rshad$,
then the universal cocycle $\smash{\vec{\cX}}$ does not belong to the coset of~$\smash{\vec{\cY}}$ mod $\ker\lshad\cP,\cdot\rshad$.
We illustrate the construction 
using two 
examples of cubic\/-\/coefficient Poisson brackets associated with the $R$-\/matrices for the Lie algebra~$\mathfrak{gl}(2)$.
\end{abstract}

\paragraph{Introduction.}
\noindent%
Bi\/-\/vector cocycles $\cQ(\cP)=\Or(\gamma)(\smash{\cP^{\otimes^n}})\in\ker\lshad\cP,\cdot\rshad$ are obtained by Kontsevich's graph orientation morphism $\Or$ from graph cocycles $\gamma$ on $n$ vertices and $2n-2$ edges in a way which is uniform 
for all finite\/-\/dimensional affine Poisson manifolds~$(M^r$,\ $\cP)$.
The (non)\/triviality of 
cocycles $\cQ(\cP)$ in the second Poisson cohomology w.r.t.\ the differential $\dd_{\cP}=\lshad\cP,\cdot\rshad$ remains an open problem, twenty\/-\/five years after the discovery of the graph complex and orientation morphism (see~\cite{Ascona96}). In all the Poisson geometries probed so far, the known infinitesimal symmetries $\dot{\cP}=\cQ(\cP)$ of the Jacobi identity $\tfrac{1}{2}\lshad\cP,\cP\rshad=0$ are $\dd_{\cP}$-\/exact: there always exists a vector field $\smash{\vec{\cY}}$ such that $\cQ(\cP)=\lshad\smash{\vec{\cY}},\cP\rshad$.
The evolution $\cP(\veps=0)\longmapsto\cP(\veps>0)$ of the tensor $\cP$ then amounts to its reparametrisations
under the diffeomorphisms of Poisson manifold which are induced by the shifts along the integral trajectories
of the vector field~$\smash{\vec{\cY}}$.
This is why, 
instead of producing new Poisson brackets from a given one, 
the Kontsevich graph flows on the spaces of Poisson bi\/-\/vectors induce (non)\/linear diffeomorphisms of the base manifold~$M$, although no more than its affine structure was the initial assumption and no possibility of smooth coordinate reparametrizations was 
presumed.

For 
a much 
used class of (scaling-)\/homogeneous Poisson 
bi\/-\/vectors $\cP=\Lie_{\vec{V}}(\cP)$, we obtain an explicit formula, $\smash{\vec{\cX}}=\Or(\gamma)(\smash{\vec{V}}\otimes\smash{\cP^{\otimes^{n-1}}})$, of a $1$-\/vector cocycle $\smash{\vec{\cX}}(\gamma,\smash{\vec{V}},\cP)\in\ker\lshad\cP,\cdot\rshad$
which is built from the graph cocycles $\gamma$ uniformly 
for all homogeneous Poisson bi\/-\/vectors $\cP$ on affine manifolds~$M^{r<\infty}$. 
The cocycle $\smash{\vec{\cX}}$ is however not necessarily a $1$-\/vector representative of the coset $\smash{\vec{\cY}}$ mod $\smash{\{\vec{\cZ}}\in\ker\lshad\cP,\cdot\rshad\}$ which would trivialise the value $\cQ(\cP)=\lshad\smash{\vec{\cY}},\cP\rshad$ of Kontsevich's 
symmetries at homogeneous Poisson structures. Indeed, the Poisson cocycle $\cQ(\cP)$ can be, we show, a nonzero bi\/-\/vector on~$M^r$, whereas the bi\/-\/vector $\lshad\smash{\vec{\cX}},\cP\rshad$ is identically zero on~$M^r$ by construction.
We contrast the formulas 
of universal cocycles $\smash{\vec{\cX}}(\gamma,\smash{\vec{V}},\cP)$ and trivialising vector fields $\smash{\vec{\cY}}$ for nonzero symmetries $\dot{\cP}=\Or(\gamma)(\cP)$ by two examples, namely,
using cubic\/-\/coefficient Poisson brackets associated with the $R$-\/matrices for $\mathfrak{gl}(2)$.

This paper is organized as follows.
In \S1 
we recall 
elements of Poisson cohomology theory in the context of Kontsevich's universal deformations of bi\/-\/vectors by using the unoriented graph cocycles.
In \S2 
we phrase 
the notion of structures which are homogeneous w.r.t.\ 
a $1$-\/vector field, and we prove the main theorem.
Finally
, we illustrate the result (cf.~\cite{OpenPrb}). 

\smallskip
\noindent%
\textbf{1. Poisson cohomology and the graph complex.}\quad
A Poisson bracket $\{\cdot,\cdot\}_\cP$ on a real manifold $M$ is a bi\/-\/linear skew\/-\/symmetric bi\/-\/derivation which takes $C^\infty(M) \times C^\infty(M) \to C^\infty(M)$ and satisfies the Jacobi identity 
$\tfrac{1}{2}\sum\nolimits_{\sigma\in S_3} \{\{\sigma(f),\sigma(g)\}_\cP,\sigma(h)\}_\cP = 0$ 
for any 
$f,g,h \in C^\infty(M)$. 
The fact that both the arguments $f,g$ and their bracket $\{f,g\}_\cP$ are scalars dictates the tensor transformation law of the components $\cP^{ij}$ of a bi-vector $\cP = \sum_{i,j} \cP^{ij}(\boldsymbol{x}) \partial_i \otimes \partial_j = \frac{1}{2}\sum_{i,j} \cP^{ij}(\boldsymbol{x})(\partial_i \otimes \partial_j - \partial_j \otimes \partial_i) = \frac{1}{2} \sum_{i,j} \cP^{ij}\, \partial_i \wedge \partial_j$ whenever the structure 
is 
referred to a 
system of coordinates $\boldsymbol{x} = (x^1,\ldots,x^r)$ and $\partial_i = \partial/\partial x^i$ is a shorthand notation.

The calculus on the space of multivectors $\Gamma(\bigwedge^\bullet TM) \cong C^\infty(\Pi T^*M)$ is simplified if one uses the parity-odd coordinates $\xi_i$ along the directions $dx^i$ in the fibres of the cotangent bundle $T^*M$ over points $\boldsymbol{a} \in M$ (which are parametrized by $x^i$).
The symbol $\xi_i$ thus corresponds to $\partial/\partial x^i$ dual to $dx^i$, and bi\/-\/vectors are $\cP = \tfrac{1}{2} \sum_{i,j} \cP^{ij} \xi_i \xi_j$, so that $\{f,g\}_\cP(\boldsymbol{a}) = (f)
\smash{\cev{\partial}}/\partial x^\mu \cdot 
\smash{\vec{\partial}}/\partial \xi_\mu \, (\cP)\, 
\smash{\cev{\partial}}/\partial \xi_\nu \cdot 
\smash{\vec{\partial}}/\partial x^\nu (g)$;
here, both the coefficients $\cP^{ij}$ and derivatives $\partial/\partial x^k$ are evaluated at the point $\boldsymbol{a} \in M$ as in the left\/-\/hand side.\footnote{The dot $\cdot$ denotes the coupling of iterated variations of the objects $f$, $\cP$, and $g$ with respect to the ca\-no\-ni\-cal\-ly conjugate variables $x^i$ and $\xi_i$, see \cite{cycle17} and references therein.}

The space of multivectors is endowed with the parity\/-\/odd Poisson bracket $\schouten{\cdot,\cdot}$ (the Schouten bracket, or antibracket) of own degree $-1$.
For arbitrary multivectors $\cP, \cQ$, the formula is 
$\schouten{\cP,\cQ} = (\cP) 
\smash{\cev{\partial}}/\partial \xi_i \cdot 
\smash{\vec{\partial}}/\partial x^i (\cQ) - (\cP) 
\smash{\cev{\partial}}/\partial x^i \cdot 
\smash{\vec{\partial}}/\partial \xi_i (\cQ)$;
in particular, $\schouten{\smash{\vec{\cX}},\smash{\vec{\cY}}} = [\smash{\vec{\cX}},\smash{\vec{\cY}}]$ is the usual commutator of vector fields $\smash{\vec{\cX}}$, $\smash{\vec{\cY}}$ on $M$.
The Schouten bracket $\schouten{\cdot,\cdot}$ is shifted\/-\/graded skew\/-\/symmetric: $\schouten{\cQ,\cP} = -(-)^{(|\cP|-1)\cdot(|\cQ|-1)} \schouten{\cP,\cQ}$ for $\cP$ and $\cQ$ grading\/-\/homogeneous.
This is why, unlike the tautology $\schouten{\smash{\vec{\cX}},\smash{\vec{\cX}}} \equiv 0$, the equation $\schouten{\cP,\cP} = 0$ is a nontrivial restriction for bi\/-\/vectors $\cP$, containing 
the tri\/-\/vector in the l.-h.s.\ 
of the Jacobi identity $\frac{1}{2}\schouten{\cP,\cP}(f,g,h) = 0$ for the bracket $\{f,g\}_\cP = \schouten{\schouten{f,\cP},g}$. 
The Schouten bracket itself satisfies the graded Jacobi identity $\schouten{\cP,\schouten{\cQ,\cR}} - (-)^{(|\cP|-1)\cdot(|\cQ|-1)}\schouten{\cQ, \schouten{\cP,\cR}} = \schouten{\schouten{\cP,\cQ},\cR}$ with $\cP$ and $\cQ$ grading-homogeneous.
This identity implies that for Poisson bi\/-\/vectors $\cP$, their adjoint action by $\partial_\cP = \schouten{\cP,\cdot}$ is a differential of degree $+1$ on the space of multivectors on $M$.
The Poisson differential $\partial_\cP$ gives rise to the Poisson cohomology $H^i_\cP(M)$ of the manifold M (see \cite{Lichnerowicz77}).\footnote{%
The group $H^0_\cP(M)$ spans the Casimirs, i.e. the functions which Poisson-commute with any $f \in C^\infty(M)$; the group $H^1_\cP(M)$ consists of vector fields which preserve the Poisson structure but do not amount to 
the Hamiltonian vector fields $\smash{\vec{\cX}}_h = \schouten{\cP,h}$; 
the second group $H^2_\cP(M) \ni \cQ$ contains infinitesimal symmetries $\cP \mapsto \cP + \varepsilon \cQ + \bar{o}(\varepsilon)$ of Poisson bi-vectors,
whereas the next group $H^3_\cP(M)$ stores the obstructions to formal integration $\cP \mapsto \cP(\varepsilon) = \cP + \sum_{k\geqslant 1} \varepsilon^k \cQ_{(k)}$ of infinitesimal symmetries $\cQ = \cQ_{(1)}$ to Poisson bi-vector formal power series satisfying $\schouten{\cP(\varepsilon),\cP(\varepsilon)} = 0$.} 

If a bi\/-\/vector $\cQ = \schouten{\smash{\vec{\cX}}, \cP}$ is a trivial Poisson cocycle, then it certainly is an infinitesimal symmetry of the Jacobi identity $\smash{\frac{1}{2}}\schouten{\cP,\cP} = 0$.
But the infinitesimal change $\schouten{\smash{\vec{\cX}},\cP}$ of the tensor $\cP$ then amounts to its reparametrisation under the infinitesimal change of coordinates $\mathbf{x}'(\boldsymbol{x}) \rightleftarrows \boldsymbol{x}(\mathbf{x'})$ along the integral trajectories of the vector field $\smash{\vec{\cX}}$ on the manifold $M$.
The following fact is true for all multivectors (regardless of the concept of Poisson cohomology).

\begin{prop}\label{PropChange}
Let $\boldsymbol{a} \in M$ be a point of an $r$-dimensional manifold and $\vec{\cX} \in \Gamma(TM)$ be a vector field on it.
For every $\varepsilon \in \mathcal{I} \subseteq \mathbb{R}$ such that there is the integral trajectory bringing $\boldsymbol{b}(-\varepsilon) \coloneqq \exp(-\varepsilon \vec{\cX})(\boldsymbol{a})$ to $\boldsymbol{a}$ by the $(+\varepsilon)$-shift, and for any choice of the $r$-tuple $\boldsymbol{x} = (x^1,\ldots,x^r)$ of local coordinates in a chart $U_\alpha$ around $\boldsymbol{a} \in M$ \textup{(}and for $|\varepsilon|$ small enough for the points $\boldsymbol{b}(-\varepsilon)$ to not yet run out of the chart $U_\alpha$\textup{),} introduce a new parametrization\footnote{Actually, this is a way to construct new coordinates for \emph{all} points of $M$ near $\boldsymbol{a}$ in $U_\alpha$, i.e. not only those which lie on a piece of the integral trajectory of $\vec{\cX}$ passing through $\boldsymbol{a}$.} for the point $\boldsymbol{a}$ by using the new $r$-tuple $\mathbf{x}'$.
By definition, put $\mathbf{x}'(\boldsymbol{a}) \coloneqq \boldsymbol{x}(\boldsymbol{b}(-\varepsilon))$.
Let $\boldsymbol{\Omega}$ be any multi-vector field 
near $\boldsymbol{a}$ on $M$.
Under the reparametrization $\mathbf{x}'(\boldsymbol{x})$, the speed at which the components of $\boldsymbol{\Omega}$ at the point $\boldsymbol{a}$ change in $\varepsilon$, as $\varepsilon \to 0$, equals $\frac{d}{d\varepsilon}\big|_{\varepsilon=0} \boldsymbol{\Omega}(\boldsymbol{a}) = \schouten{\vec{\cX},\boldsymbol{\Omega}}(\boldsymbol{a})$.
In particular, a $1$-vector field $\vec{\cY}$ near $\boldsymbol{a}$ would change at $\boldsymbol{a}$ as fast as its commutator with the vector field $\vec{\cX}$\textup{:} $\frac{d}{d\varepsilon}\big|_{\varepsilon=0} \vec{\cY}(\boldsymbol{a}) = [\vec{\cX},\vec{\cY}](\boldsymbol{a})$.
\end{prop}

\noindent%
The geography of the set of Poisson structures near a given bracket $\{\cdot,\cdot\}_\cP$ on a given manifold $M^r$ is, generally speaking, unknown.
All the more it was a priori unclear whether Poisson bi-vectors $\cP$, irrespective of the dimension $r \geqslant 3$, topology of $M^r$, etc., can be infinitesimally shifted by Poisson $2$-cocycles $\cQ(\cP)$, the construction of which would be universal for all $\cP$.
The discovery of the graph complex in 1993--94 allowed Kontsevich to state (in \cite{Ascona96}) the affirmative answer to the above question. 
Namely, the graph orientation morphism $\Or(\cdot)(\cP)\colon \ker \Id \ni \gamma \mapsto \cQ(\cP) \in \ker \partial_\cP$ takes graph cocycles on $n$ vertices and $2n-2$ edges in each term (e.g., the tetrahedron, cf. \cite{f16,OrMorphism,sqs17,JNMP17}) to Poisson cocycles whenever the bi-vector $\cP$ itself is Poisson.
Willwacher \cite{WillwacherGRT} revealed that the generators of Drinfeld's Gro\-then\-dieck\/-\/-Teich\-m\"ul\-ler Lie algebra $\mathfrak{grt}$ are 
source of at least countably many such cocycles in the vertex-edge bi-grading $(n,2n-2)$; these cocycles are marked by the $(2\ell+1)$-wheel graphs (e.g., see \cite{JNMP17,DolgushevRogersWillwacher}).
Brown proved in \cite{Brown2012} that, under the Willwacher isomorphism $\mathfrak{grt} \cong H^0(\textsc{Gra})$ these graph cocycles with wheels generate a free Lie subalgebra in $\mathfrak{grt}$, which means effectively that the iterated commutators of already known cocycles -- under the bracket in the differential graded Lie algebra $\textsc{Gra}$ of graphs -- would never vanish. 
The commutator of two cocycles is a cocycle by the Jacobi identity.
All of them again being of the bi-grading $(n,2n-2)$, these graph cocycles determine countably many infinitesimal symmetries of a given Poisson bi-vector $\cP$; the construction is uniform 
for all the geometries $(M^r, \cP)$.

\begin{lemma}\label{LemmaCommutatorGraSym}
For a given Poisson bi-vector $\cP$, the graph orientation mapping $\Or(\cdot)(\cP)\colon \ker \Id \ni \gamma \mapsto \cQ(\cP) \in \ker \partial_\cP$ is a Lie algebra morphism that takes the bracket of two cocycles in bi-grading $(n,2n-2)$ to the commutator $[\frac{d}{d\varepsilon_1}, \frac{d}{d\varepsilon_2}](\cP)$ of two symmetries $\frac{d}{d\varepsilon_i}(\cP) = \cQ_i(\cP)$.\footnote{%
By Brown 
\cite{Brown2012}, the commutator does in general not vanish for Willwacher's odd-sided wheel cocycles.}
\end{lemma}

\noindent%
By construction, the components of universal symmetry bi-vectors $\cQ(\cP)$ are differential polynomials w.r.t.\ 
the components $\cP^{ij}$ of the Poisson bi-vector $\cP$ that evolves.
It can of course be that a graph flow $\dot{\cP} = \Or(\gamma)(\cP)$ vanishes identically over the manifold $M^r$ whenever $\cQ$ is evaluated at a particular class of Poisson structures $\cP$.\footnote{%
\textbf{Example.} 
So it is for the Kontsevich tetrahedral flow (\cite{Ascona96} and \cite{f16
}) evaluated at the Kirillov\/--\/Kostant linear Poisson brackets on the duals $\mathfrak{g}^*$ of Lie algebras because in every term within the cocycle $\cQ(\cP)$ under study, at least one copy is $\cP$ is differentiated at least twice with respect to the global coordinates on $\mathfrak{g}^*$.}
Nevertheless, there is no 
mechanism which would force a given Kontsevich's graph flow to vanish at all Poisson structures on all manifolds of all dimensions.%
\footnote{\label{Foot3DdetPoisson}\textbf{Example.}  
The Poisson bi-vectors $\cP = \Id a_1\wedge\ldots\wedge \Id a_m/\dvol(\mathbb{R}^{m+2})$ of Nambu type with arbitrary Casimirs $a_1,\ldots,a_m \in C^\infty(\mathbb{R}^{m+2})$ and an arbitrary density in the volume element can have polynomial components $\cP^{ij} \in \mathbb{R}[x^1,\ldots,x^{m+2}]$ of degrees as high as need be w.r.t.\ 
the global Cartesian coordinates $x^\alpha$ on the vector space $\mathbb{R}^{m+2}$.
The universal symmetries $\dot{\cP} = \Or(\gamma)(\cP)$ obtained from Kontsevich's graph cocycles deform the symplectic foliation (which is given in $\mathbb{R}^{m+2}$ by the intersections of the level sets for the Casimirs $a_1,\ldots,a_m$) in a regular way on an open dense subset of $\mathbb{R}^{m+2}$, so that the symmetries $\dot{\cP} = \cQ(\cP)$ preserve this Nambu class of Poisson brackets: the flows force the evolution of the Casimirs and the volume density. Its integrability is an open problem; by Lemma \ref{LemmaCommutatorGraSym} and \cite{Brown2012}, the evolutions induced by different graph cocycles do not commute.
}
Independently, it remains an open problem (cf. \cite{OpenPrb}) whether there is a Poisson manifold $(M^r, \cP)$ and a graph cocycle $\gamma$ such that the Poisson cohomology class of 
$\cQ(\cP) \mathrel{{:}{=}} 
\Or(\gamma)(\cP)$ would be \emph{nontrivial} in $H^2_\cP(M)$.
In other words, for all the shifts $\cQ = \Or(\gamma)$ and all Poisson bi\/-\/vectors tried so far, the Poisson coboundary equation $\cQ(\cP) = \schouten{\smash{\vec{\cX}},\cP}$ did have vector field solutions $\smash{\vec{\cX}}$ on the 
manifolds~$M$.

\begin{rem}
Obtained from the graphs $\gamma \in \ker \Id$, the 
symmetries $\cQ(\cP) = \Or(\gamma)(\cP) \in \ker \schouten{\cP,\cdot}$ are independent of a choice of local coordinates $x^i$ (hence $\xi_i$) on a chart if, the Kontsevich construction requires, the manifold $M^r$ is endowed with an \emph{affine} structure: all the coordinate transformations amount to $\mathbf{x}' = A\boldsymbol{x}+\smash{\vec{\boldsymbol{b}}}$ with a constant (over the intersection of charts) Jacobian matrix $A$.
The parity-odd fibre variables are transformed
using the inverse Jacobian matrix, $\xi_i = A^{i'}_{i} \xi'_{i'}$, making sense of the couplings $\cev{\partial}/\partial \xi_i \cdot \vec{\partial}/\partial x^i$ which decorate the oriented edges of Kontsevich's graphs after the morphism $\Or$ works (see \cite{OrMorphism,Ascona96}
). 
The problem of Poisson cohomology class (non)triviality for the Kontsevich infinitesimal symmetries $\dot{\cP} = \cQ(\cP) \in \ker \schouten{\cP,\cdot}$ thus acquires two diametrally opposite interpretations:\\[1pt]
\mbox{ }\quad\textbf{\textit{1}} (as in \cite{Ascona96})\textbf{.}\quad 
The Poisson manifold $M^{r<\infty}$ is equipped with \emph{both} the smooth and affine structures.\footnote{On the circle $\mathbb{S}^1$, the affine coordinate `angle' is obvious whereas the smooth structure is used in the realm of Poincar\'e topology. A smooth atlas is always available for the spheres $\mathbb{S}^r$, but not for all $r \in \mathbb{N}$ would the $r$-dimensional sphere admit an affine structure.
}
By definition, two Poisson bi-vectors are equivalent, $\cP_1 \sim \cP_2$, if they are related by a diffeomorphism of the manifold $M$: using its smooth structure, the diffeomorphism iden\-ti\-fies points in two copies of $M$, then relating the Poisson tensors by local coordinate re\-pa\-ra\-met\-ri\-za\-ti\-ons near the respective points.
The affine structure on $M$ is now used to run the Kontsevich flows in two initial value 
problems $\dot{\cP_i}(\varepsilon) = \cQ(\cP_i(\varepsilon))$, $\cP_i(\varepsilon=0) = \cP_i$.
The Poisson triviality $\cQ(\cP(\varepsilon)) = \schouten{\vec{\cX}(\varepsilon), \cP(\varepsilon)}$ would relate either of bi-vectors $\cP_i(\varepsilon)$ back to the Cauchy datum $\cP_i$ by diffeomorphisms (as long as $|\varepsilon|$ is small enough).
Consequently, the Poisson bi-vectors $\cP_1(\varepsilon) \sim \cP_2(\varepsilon)$ do not run out of the old equivalence class.
In conclusion, the 
goal is to produce essentially new Poisson brackets by using a 
nontrivial cocycle $\cQ$, two given structures on the manifold $M^r$, and its diffeomorphism. 
No examples of nontrivial action, so that $\cP_2(\varepsilon) \not\sim \cP_i 
\not\sim \cP_1(\varepsilon)$ at $\varepsilon > 0$, have ever been produced since 1996 (see~\cite{DolgushevRogersWillwacher,Ascona96}). 
\\
\mbox{ }\quad\textbf{\textit{2}} (as in \cite{OpenPrb})\textbf{.}\quad 
The Poisson manifold $M^{r<\infty}$ is equipped only with an affine structure.
The countably many $\mathfrak{grt}$-related graph cocycles on $n$ vertices and $2n-2$ edges in every term (the tetrahedron, the pentagon\/-\/wheel cocycle, etc., see \cite{JNMP17,
WillwacherGRT}) 
generate a noncommutative Lie algebra of infinitesimal symmetries $\cQ(\cP) = \Or(\gamma)(\cP)$ for a given Poisson structure $\cP$.
Consider the extreme case when \emph{all} the cocycles $\cQ(\cP) \in \ker \schouten{\cP,\cdot}$ are 
exact in the cohomology group $H^2_\cP(M)$ w.r.t.\ the Poisson differential $\partial_\cP$.
This assumption 
gives rise to the countable set of vector fields $\smash{\vec{\cY}}(\gamma,\cP)$ on $M$ such that $\cQ(\cP) = \schouten{\smash{\vec{\cY}},\cP}$.
(Some of these vector fields can be identically zero over $M$.)
But if at least one such vector field is not constant w.r.t.\ the affine structure on $M$, then the shifts along its integral trajectories are nonlinear diffeomorphisms of $M$.
The evolution of bi-vector $\cP$ is $\dot{\cP} = \cQ(\cP) = \schouten{\smash{\vec{\cY}},\cP}$ or similarly, $\dot{\Omega} = \schouten{\smash{\vec{\cY}},\Omega}$ for any multi-vector $\Omega$ on $M$ (see Proposition~\ref{PropChange}); this evolution is now seen as mutlivectors' response to the diffeomorphism whose construction refers only to the simple, affine local portrait of $M$.
Summarizing, the store of flows $\Or(\gamma)(\cP)$ from the $\mathfrak{grt}$-related graph cocycles $\gamma$ could be enough to approximate arbitrary smooth vector fields on $M^r$, that is, imitate its smooth structure. 
Whether this theoretical possibility is actually realised in relevant Poisson models is an open problem.
\end{rem}

The Kontsevich symmetry construction is, therefore, either a generator of new Poisson brackets or the mechanism that provides diffeomorphisms of the underlying manifold.

\smallskip
\noindent%
\textbf{2. Homogeneous Poisson structures.}\quad
By definition, a bi\/-\/vector $\cP$ on a manifold $M$ is called \emph{homogeneous} (of 
scale 
$\lambda$) with respect to a vector field $\smash{\vec{\cV}}$ on $M$ if $\schouten{\smash{\vec{\cV}},\cP} = \lambda \cdot \cP$.

\begin{example}
Let $M = \mathbb{R}^r$ be a vector space (only linear reparametrizations $\mathbf{x}' = A\boldsymbol{x}$ are allowed, so that the polynomial degrees of monomials in the ring $\mathbb{R}[x^1,\ldots,x^r]$ is well defined).
Introduce the Euler vector field $\smash{\vec{\cE}} = \smash{\sum_{i=1}^r} x^i\,\partial/\partial x^i$, and let all the components $\cP^{ij}$ of a bi-vector $\cP$ be homogeneous polynomials of degree $d$ in the variables $x^i$.
Then we have that $\schouten{\smash{\vec{\cE}},\cP} = (d-2)\cdot\cP$, which means that $\cP$ is homogeneous of scale 
$d-2$ w.r.t.\ the Euler vector field $\smash{\vec{\cE}}$.
In particular, if $d \neq 2$ (i.e.\ if the coefficients of bi-vector $\cP$ are not quadratic), then we set $\smash{\vec{\cV}} = (d-2)^{-1} \cdot \smash{\vec{\cE}}$ and from the equality $\cP = \schouten{\smash{\vec{\cV}},\cP}$ we obtain that the same bi-vector $\cP$ has homogeneity scale 
$\lambda=1$ w.r.t.\ 
the multiple $\smash{\vec{\cV}}$ of the Euler vector field $\smash{\vec{\cE}}$ on~
$\mathbb{R}^r$.
\end{example}

\begin{example}
Under the same assumptions, suppose further that $\gamma = \sum_a c_a \gamma_a 
$ is a graph cocycle with $n$ vertices and $2n-2$ edges in every term $\gamma_a$ (e.g., take the tetrahedron).
Orient the ordered (by First $\prec\cdots\prec$ Last) edges in every $\gamma_a$ using the edge decoration operators $\smash{\vec{\Delta}}_{ij} = \smash{\sum_{\mu=1}^r} (\smash{\vec{\partial}}/\partial \xi_\mu^{(i)} \otimes \smash{\vec{\partial}}/\partial x^\mu_{(j)} + \smash{\vec{\partial}}/\partial\xi_\mu^{(j)} \otimes \smash{\vec{\partial}}/\partial x^\mu_{(i)})$.
By placing a copy of bi\/-\/vector $\cP = \smash{\frac{1}{2}} \cP^{kl}(\boldsymbol{x})\xi_k\xi_l$ in each vertex $\smash{v^{(i)}}$ of $\gamma_a$ and taking the sum (over the graph index $a$) of products of the content of vertices in $\gamma_a$ after all the edge operators $\smash{\vec{\Delta}_{ij}}$ work, we obtain\footnote{We refer to the original paper \cite{Ascona96} and to \cite{OrMorphism} 
for illustrations and discussion how the graph orientation morphism works in practice.} the bi-vector $\cQ(\cP) \mathrel{{:}{=}} \Or(\gamma)(\cP)$.
Then the coefficients of the bi\/-\/vector $\cQ(\cP)$ are homogeneous polynomials of degree $n\cdot d - (2n-2)$ with respect to $x^1$, $\ldots$, $x^r$, so that $\schouten{\smash{\vec{\cE}}, \cQ(\cP)} = n(d-2) 
\cQ(\cP)$.
In particular, if $d \neq 2$, then $\schouten{\smash{\vec{\cV}},\cQ(\cP)} = n \cdot \cQ(\cP)$, whereas quadratic\/-\/coefficient bi\/-\/vectors $\cP$ (with $d=2$) are deformed within their subspace by the quadratic bi\/-\/vectors $\cQ(\cP)$ which are obtained from the Kontsevich graph cocycles.
\end{example}

\begin{lemma}\label{LemmaVertexNumber}
If a Poisson bi\/-\/vector $\cP=\schouten{\smash{\vec{\cV}},\cQ(\cP)}$ is homogeneous and $\cQ(\cP)=\Or(\gamma)(\smash{\cP^{\otimes^{n}}})$ is built from a graph cocycle~$\gamma$ on $n$ vertices, now containing a copy of $\cP$ in each vertex, then the bi\/-\/vector $\cQ(\cP)$ is also homogeneous\textup{:} $\schouten{\smash{\vec{\cV}},\cQ(\cP)} = n \cdot \cQ(\cP)$, 
so that its scale is
~$n$.\footnote{The proof amounts to the Leibniz rule: let us inspect how fast the bi\/-\/vector $\cQ(\cP)$, which by construction is a homogeneous differential polynomial of degree $n$ in~$\cP$, evolves along 
the vector field~$\smash{\vec{\cV}}$.}
\end{lemma}

\begin{rem}[{\cite[Rem.~4.9]{PichereauIsolated}}]
Consider a Nambu\/-\/type Poisson bi-vector $\cP = \Id a/ \Id x \Id y \Id z$ on $\mathbb{R}^3$ with Cartesian coordinates $x,y,z$\textup{;} here $a \in \mathbb{R}[x,y,z]$ is a weight\/-\/homogeneous polynomial with an isolated singularity at the origin%
\footnote{The Milnor number is the dimension 
$\dim_\mathbb{R} \mathbb{R}[x,y,z]/(\partial a/\partial x, \partial a/\partial y, \partial a/\partial z)$ 
-- 
here, $<\infty$ by assumption
.}%
\textup{,} so that $(w_{(x)} \cdot x \dd/\dd x + w_{(y)} \cdot y \dd/\dd y + w_{(z)} \cdot z \dd/\dd z)(a) = w_{(a)} \cdot a$.
Then a vector field $\smash{\vec{\cV}}$ with polynomial components satisfying the first\/-\/order PDE $\cP = \schouten{\smash{\vec{\cV}},\cP}$ exists if and only if\,{}\footnote{This means that not all Nambu\/-\/type Poisson bi-vectors $\cP = \Id a/ \Id x \Id y \Id z$ are homogeneous w.r.t.\ a vector field $\smash{\vec{\cV}}$ with polynomial components; the PDE $\cP = \schouten{\smash{\vec{\cV}},\cP}$ with polynomial coefficients and unknown $\smash{\vec{\cV}}$ can in principle admit non\/-\/polynomial solutions.}
the weight degree $w_{(a)}$ of the polynomial $a$ is not equal to the sum $w_{(x)} + w_{(y)} + w_{(z)}$ of weight degrees for the variables $x,y,z$.\footnote{\textbf{Example.}\ %
If the weights of $(x,y,z)$ are $(1,1,1)$ and $a = \smash{\frac{1}{3}}(x^3+y^3+z^3)$ is cubic\/-\/homogeneous, then the components of Poisson bi-vector $\cP$ are quadratic and (by the above and also by \cite[Exerc.~4.5.7c]{VanhaeckeBookPoisson}) not of the form $\cP = \schouten{\smash{\vec{\cV}}, \cP}$ for any polynomial\/-\/coefficient vector field $\smash{\vec{\cV}}$.
The non\/-\/existence of a solution $\smash{\vec{\cV}}$ with smooth non\/-\/polynomial coefficients is a separate problem.}
\end{rem}


Summarizing, the homogeneity assumption about 
bi\/-\/vectors $\cP$ is restrictive; it is not always satisfied in Poisson models.

\begin{theor}
Let $(M,\cP)$ be an affine finite\/-dimensional real Poisson manifold 
with $\cP=\lshad\vec{\cV},\cP\rshad$ homogeneous.
Let $\gamma=\sum_a c_a\cdot\gamma_a$ be a graph cocycle
consisting of unoriented graphs~$\gamma_a$ over $n$ vertices and $2n-2$ edges
\textup{(}with a fixed ordering of edges in each~$\gamma_a$\textup{).}
Then the $1$-vector $\smash{\vec{X}}(\gamma,\vec{\cV},\cP) = \Or(\gamma)(\smash{\vec{\cV}\otimes\cP^{\otimes^{n-1}}})$,
which is obtained by representing each edge $i{-}\!\!{-}j$ with the operator $\smash{\vec{\Delta}_{ij}}$ and 
by \textup{(}graded-\textup{)}\/symmetrizing over all the ways 
$\sigma\in\mathbb{S}_n$ to send 
the $n$-\/tuple $\smash{\vec{\cV}\otimes\cP^{\otimes^{n-1}}}$ into the $n$ vertices in each~$\gamma_a$,
is a Poisson cocycle\textup{:}
$\smash{\vec{\cX}}\in\ker\lshad\cP,\cdot\rshad$.\footnote{
\textbf{Open problem} (for $\cP$ homogeneous and Poisson
)\textbf{.} 
Is
the universal 
$1$-\/vector field 
$\smash{\vec{\cX}}(\gamma,\vec{\cV},\cP)\in\ker\dd_{\cP}$ 
Hamiltonian, i.e.\ $\smash{\vec{\cX}}=\lshad\cP,h\rshad$ for $h\in C^\infty(M)$ or at least, $\smash{\vec{\cX}}=\cP\inner\eta$
for a maybe not exact $1$-\/form $\eta$ on $M$\,?
}\\
\mbox{ }\quad
The vector field $\smash{\vec{\cX}}$ 
is defined up to adding arbitrary Poisson $1$-\/cocycles $\smash{\vec{\cZ}}\in\ker\lshad\cP,\cdot\rshad$.%
\end{theor} 

\begin{proof}
The expansion $0=\Or(\Id\gamma)(\smash{\vec{\cV}}\otimes\smash{\cP^{\otimes^n}})$ for $\gamma\in\ker\Id$
goes along the lines of~\cite{Ascona96} and~\cite{OrMorphism,DolgushevRogersWillwacher,Jost2013}, 
but the $(n+1)$-\/tuple of multivectors now contains one $1$-\/vector and only $n$ copies of the Poisson
bi\/-\/vector $\cP$. By assumption, $\Id\gamma=\boldsymbol{0}\in{}$\textsc{Gra};
recall that $\Or(\boldsymbol{0})(\text{any multivectors})=0\in\Gamma(\bigwedge^\bullet TM)$.
This zero l.-h.s.\ equates $0=\bigl(\pi_S \mathbin{\smash{ \vec{\circ} }} \Or(\gamma) 
- (-)^{(-1)\cdot(-N)} \Or(\gamma) \mathbin{\smash{ \vec{\circ} }} \pi_S\bigr) 
(\smash{\vec{\cV}}\otimes\smash{\cP^{\otimes^n}})$.\footnote{Here, $\pi_S$ is the 
graded\/-\/symmetric Schouten bracket (so $\pi_S(F,G)=(-)^{|F|-1}\lshad F,G\rshad$), the graph insertion $\smash{ \vec{\circ} }$ into vertices is now the endomorphism insertion into argument slots, $|\pi_S|=-1$, and $N=2n-2$ is the even number of edges in $\gamma$, hence minus the even number of $\dd/\dd\xi_\mu$ in the edge operators $\smash{\vec{\Delta}_{ij}}$ making $\Or(\gamma)$.}

The appointment of graded (multi)\/vectors into the vertices of $\Id\gamma$
(hence, into the argument slots of the endomorphism $\Or(\Id\gamma)$)
is achieved by the graded symmetrization using
$((n+1)!)^{-1}\,\Or(\Id\gamma)\bigl(\pm\sigma(\smash{\vec{\cV}}\otimes\smash{\cP^{\otimes^n}})\bigr)$.
Fortunately, the field $\smash{\vec{\cV}}$ is the only parity\/-\/odd object,
so its transpositions with the parity\/-\/even bi\/-\/vectors $\cP$ produce no sign factor:
these $\pm$ are all $+$. Likewise, the $n!$ permutations of $n$ indistinguishable copies of $\cP$
leave only $n+1$ from $(n+1)!$ in the denominator; to get rid of it, let us multiply by $n+1$ both sides of the equality
$0=\Or(\Id\gamma)(\smash{\vec{\cV}}\otimes\smash{\cP^{\otimes^n}})$.
The symmetrization thus amounts, by the linearity of $\Or(\gamma)$, to its evaluation at the sum of arguments,
$\smash{\vec{\cV}}\cdot\cP^n + \cP\cdot\smash{\vec{\cV}}\cdot\cP^{n-1}+\ldots+\cP^n\cdot\smash{\vec{\cV}}$,
in which the ordering of (multi)\/vectors now matches an arbitrary fixed enumeration of the vertices.

The rest of the proof is standard.\footnote{We have
$
0= \Or(\gamma)(\pi_S(\smash{\vec{\cV}},\cP),\cP^{n-1}) + \Or(\gamma)(\pi_S(\cP,\smash{\vec{\cV}}),\cP^{n-1})
 + \Or(\gamma)(\pi_S(\cP,\cP),\smash{\vec{\cV}},\cP^{n-2}) +\ldots + \Or(\gamma)(\pi_S(\cP,\cP),\cP^{n-2},\smash{\vec{\cV}})
+\Or(\gamma)(\smash{\vec{\cV}},\pi_S(\cP,\cP),\cP^{n-2}) + \Or(\gamma)(\cP,\pi_S(\smash{\vec{\cV}},\cP),\cP^{n-2})
+ \Or(\gamma)(\cP,\pi_S(\cP,\smash{\vec{\cV}}),\cP^{n-2}) + \Or(\gamma)(\cP,\pi_S(\cP,\cP),\smash{\vec{\cV}},\cP^{n-3})
+ \ldots + \Or(\gamma)(\cP,\pi_S(\cP,\cP),\cP^{n-3},\smash{\vec{\cV}})
+ \ldots \text{(the Schouten bracket $\pi_S$ passes along the slots towards the end)}
+ \Or(\gamma)(\smash{\vec{\cV}},\cP^{n-2},\pi_S(\cP,\cP)) + \ldots + \Or(\gamma)(\cP^{n-2},\smash{\vec{\cV}},\pi_S(\cP,\cP))
+ \Or(\gamma)(\cP^{n-1},\pi_S(\smash{\vec{\cV}},\cP)) + \Or(\gamma)(\cP^{n-1},\pi_S(\cP,\smash{\vec{\cV}}))
\mathrel{\mathbf{-}} (-)^N\cdot\bigl[ \pi_S\bigl(\Or(\gamma)(\smash{\vec{\cV}}\cdot\cP^{n-1} + \cP\cdot\smash{\vec{\cV}}\cdot\cP^{n-2}
   + \ldots + \cP^{n-1}\cdot\smash{\vec{\cV}}),\cP\bigr)
+ \pi_S(\Or(\gamma)(\cP^n),\smash{\vec{\cV}}) + \pi_S(\smash{\vec{\cV}},\Or(\gamma)(\cP^n)) 
+ \pi_S\bigl(\cP,\Or(\gamma)(\smash{\vec{\cV}}\cdot\cP^{n-1} + \cP\cdot\smash{\vec{\cV}}\cdot\cP^{n-2}
   + \ldots + \cP^{n-1}\cdot\smash{\vec{\cV}})\bigr) \bigr]  $.
For $\cP$ Poisson, $\pi_S(\cP,\cP)=0$, so we exclude all such terms (
\cite{sqs15}).
The remaining graded\/-\/symmetric Schouten brackets $\pi_S$ contain a bi\/-\/vector as one of the arguments, hence those can be swapped at no sign factor; all doubles, so let us divide by~$2$.
}
There remains $0=\Or(\gamma)\bigl( \pi_S(\smash{\vec{\cV}},\cP)\cdot\cP^{n-1}) + \cP\cdot\pi_S(\smash{\vec{\cV}},\cP)\cdot\cP^{n-2}) 
+ \ldots + \cP^{n-1}\cdot\pi_S(\smash{\vec{\cV}},\cP) \bigr)
- (-)^N \bigl[ \pi_S\bigl( \Or(\gamma)\bigl( \smash{\vec{\cV}}\cdot\cP^{n-1} + \cP\cdot\smash{\vec{\cV}}\cdot\cP^{n-2}
+ \ldots + \cP^{n-1}\cdot\smash{\vec{\cV}} \bigr), \cP\bigr) + \pi_S(\Or(\gamma)(\cP^n),\smash{\vec{\cV}}) \bigr] $.
By the homogeneity assumption, $\pi_S(\smash{\vec{\cV}},\cP)=(-)^{1-1}\lshad\smash{\vec{\cV}},\cP\rshad=\cP$, 
and by construction, $\Or(\gamma)(\cP^n)=\cQ(\cP)$, whence the minuend equals $n\cdot\cQ(\cP)$.
By Lemma~\ref{LemmaVertexNumber}, the graph flow is also homogeneous: $\lshad\smash{\vec{\cV}},\cQ(\cP)\rshad=\lambda\cdot\cQ(\cP)$ with the vertex count $\lambda=n$.
We obtain the equality
\begin{multline*}
(-)^{2n-2}\cdot\lshad \Or(\gamma)\bigl(\smash{\vec{\cV}}\cdot\cP^{n-1} + \cP\cdot\smash{\vec{\cV}}\cdot\cP^{n-2} +\ldots+
\cP^{n-1}\cdot\smash{\vec{\cV}}\bigr),\cP\rshad ={}\\
{}= n\cdot\cQ(\cP) - (-)^{2n-2}\lambda\cdot\cQ(\cP) = (n-(-)^{\text{even}}\,n)\cdot\cQ(\cP)\equiv 0.
\end{multline*}
We conclude that the $1$-\/vector $\smash{\vec{\cX}} \mathrel{{:}{=}} \Or(\gamma)(\smash{\vec{\cV}}\otimes\smash{\cP^{\otimes^{n-1}}})$ lies in~$\ker\lshad\cP,\cdot\rshad 
$.\footnote{\textbf{Exercise.} Extend the proof 
to the case $n=1$, 
$\gamma={\bullet}$, $\Id\gamma=-{\bullet}\!{-}\!\!{-}\!{\bullet}$ (so that the l.-h.s.\ was nonzero
)%
.}
\end{proof}

\begin{example}
Take the Lie algebra $
\mathfrak{gl}_2(\mathbb{R})$ with its four\/-\/dimensional vector space structure;
denote by $x$, $y$, $z$, $v$ the Cartesian coordinates.
Consider the $R$-\/matrix $\left(\begin{smallmatrix} x & y \\ z & v\end{smallmatrix}\right) 
\mapsto \left(\begin{smallmatrix} 0 & y \\ -z & 0 \end{smallmatrix}\right)$
known from~\cite{VanhaeckeBookPoisson};
the standard construction then yeilds the Poisson bi\/-\/vector in the algebra of coordinate functions,
$\cP = \left( x^{2} y + y^{2} z \right) 
{\partial_x }
\wedge 
{\partial_y } + \left( x^{2} z + y z^{2} \right) 
{\partial_x }\wedge 
{\partial_z } + \left( 2 \, x y z + 2 \, y z v \right) 
{\partial_x }\wedge 
{\partial_v } + \left( y^{2} z + y v^{2} \right) 
{\partial_y }\wedge 
{\partial_v } + \left( y z^{2} + z v^{2} \right) 
{\partial_z }\wedge 
{\partial_v }
$.
This bracket has cubic\/-\/nonlinear homogeneous polynomial coefficients, hence $d=3$.
The vector field $\smash{\vec{\cV}} = (d-2)^{-1}\cdot\smash{\vec{\cE}}$ is the (multiple of the) Euler vector field on~$\mathbb{R}^{4}$.
As the graph cocycle~$\gamma$, we take the tetrahedron (see~\cite{f16,Ascona96});
then the symmetry flow is
$\dot{\cP}=\cQ(\cP) = 
( -48 \, x^{5} y - 288 \, x^{3} y^{2} z - 240 \, x y^{3} z^{2} + 192 \, y^{3} z^{2} v - 384 \, x y^{2} z v^{2} - 192 \, y^{2} z v^{3} ) 
{\partial_x }\wedge 
{\partial_y } + ( -48 \, x^{5} z - 288 \, x^{3} y z^{2} - 240 \, x y^{2} z^{3} + 192 \, y^{2} z^{3} v - 384 \, x y z^{2} v^{2} - 192 \, y z^{2} v^{3} ) 
{\partial_x }\wedge 
{\partial_z } + ( -336 \, x^{4} y z - 480 \, x^{2} y^{2} z^{2} - 576 \, x^{3} y z v + 480 \, y^{2} z^{2} v^{2} + 576 \, x y z v^{3} + 336 \, y z v^{4} ) 
{\partial_x }\wedge 
{\partial_v } + ( 192 \, x^{3} y^{2} z - 192 \, x y^{3} z^{2} + 288 \, y^{2} z v^{3} + 48 \, y v^{5} + 48 \, {(8 \, x^{2} y^{2} z + 5 \, y^{3} z^{2})} v ) 
{\partial_y }\wedge 
{\partial_v } + ( 192 \, x^{3} y z^{2} - 192 \, x y^{2} z^{3} + 288 \, y z^{2} v^{3} + 48 \, z v^{5} + 48 \, {(8 \, x^{2} y z^{2} + 5 \, y^{2} z^{3})} v ) 
{\partial_z }\wedge 
{\partial_v }
$.
We detect that this bi\/-\/vector is a coboundary, $\cQ(\cP)=\lshad\smash{\vec{\cY}},\cP\rshad$ with the vector
$\vec{\cY} = 
( -24 \, x^{4} + 120 \, y^{2} z^{2} - 96 \, y z v^{2} ) 
{\partial_x } + ( 96 \, x^{3} y - 96 \, y v^{3} ) 
{\partial_y } + ( 96 \, x^{3} z - 96 \, z v^{3} ) 
{\partial_z } + ( 96 \, x^{2} y z - 120 \, y^{2} z^{2} + 24 \, v^{4} ) 
{\partial_v }
$ mod $\ker\lshad\cP,\cdot\rshad$.
The vector field $\smash{\vec{\cY}}\notin\ker\dd_{\cP}$ cannot be Poisson\/-\/exact (clearly, $\cQ(\cP)\not\equiv 0$),
hence $\smash{\vec{\cY}}$ does not mark the Poisson cocycle of zero $1$-\/vector.\footnote{%
Likewise, by using another $R$-\/matrix 
for $\mathfrak{gl}_2(\mathbb{R})$, namely
$\left(\begin{smallmatrix} x & y \\ z & v \end{smallmatrix}\right) \mapsto 
\left(\begin{smallmatrix} x & y \\ -z & v \end{smallmatrix}\right)$ also from~\cite{VanhaeckeBookPoisson},
we obtain the Poisson bi\/-\/vector
$\cP = 
2 \, x^{2} y 
{\partial_x }\wedge 
{\partial_y } + 2 \, y z^{2} 
{\partial_x }\wedge 
{\partial_z } + ( 2 \, x y z + 2 \, y z v ) 
{\partial_x }\wedge 
{\partial_v } + ( -2 \, x y z + 2 \, y z v ) 
{\partial_y }\wedge 
{\partial_z } + 2 \, y v^{2} 
{\partial_y }\wedge 
{\partial_v } + 2 \, y z^{2} 
{\partial_z }\wedge 
{\partial_v }
$
on $\mathbb{R}^4$ with Cartesian coordinates $x$, $y$, $z$, $v$.
The tetrahedral flow then equals
$\dot{\cP}=\cQ(\cP) = 
( -384 \, x^{5} y - 384 \, x^{3} y^{2} z - 1536 \, x y^{2} z v^{2} + 384 \, { (x^{2} y^{2} z - 4 \, y^{3} z^{2} )} v ) 
{\partial_x }\wedge 
{\partial_y } + ( -384 \, x^{3} y z^{2} - 2688 \, x y^{2} z^{3} + 1152 \, x y z^{2} v^{2} + 384 \, y z^{2} v^{3} - 384 \, { (3 \, x^{2} y z^{2} - 7 \, y^{2} z^{3} )} v ) 
{\partial_x }\wedge 
{\partial_z } + ( -384 \, x^{4} y z - 2688 \, x^{2} y^{2} z^{2} - 1536 \, x^{3} y z v + 2688 \, y^{2} z^{2} v^{2} + 1536 \, x y z v^{3} + 384 \, y z v^{4} ) 
{\partial_x }\wedge 
{\partial_v } + ( 384 \, x^{4} y z + 384 \, x^{2} y^{2} z^{2} + 1536 \, y^{3} z^{3} - 384 \, x y z v^{3} + 384 \, y z v^{4} + 384 \, {(x^{2} y z + y^{2} z^{2} )} v^{2} - 384 \, { (x^{3} y z - 2 \, x y^{2} z^{2} )} v ) 
{\partial_y }\wedge 
{\partial_z } + ( 1536 \, x y^{3} z^{2} + 1536 \, x^{2} y^{2} z v - 384 \, x y^{2} z v^{2} + 384 \, y^{2} z v^{3} + 384 \, y v^{5} ) 
{\partial_y }\wedge 
{\partial_v } + ( -384 \, x^{3} y z^{2} - 2688 \, x y^{2} z^{3} + 1152 \, x y z^{2} v^{2} + 384 \, y z^{2} v^{3} - 384 \, { (3 \, x^{2} y z^{2} - 7 \, y^{2} z^{3} )} v ) 
{\partial_z }\wedge 
{\partial_v }
$.
It is Poisson\/-\/trivial: $\cQ(\cP)=\lshad\smash{\vec{\cY}},\cP\rshad$ with a representative
$\vec{\cY} = 
( -96 \, x^{4} + 576 \, y^{2} z^{2} - 384 \, y z v^{2} ) 
{\partial_x } + ( -192 \, x y^{2} z + 192 \, y^{2} z v - 384 \, y v^{3} ) 
{\partial_y } + ( -96 \, x^{3} z - 96 \, x z v^{2} + 96 \, z v^{3} + 96 \, {\left(x^{2} z - 4 \, y z^{2}\right)} v ) 
{\partial_z } + ( -576 \, y^{2} z^{2} - 384 \, x y z v + 96 \, v^{4} ) 
{\partial_v }
$.
These explicit examples of Poisson\/-\/exact bi\/-\/vector flows $\dot{\cP}=\cQ(\cP)=\lshad\smash{\vec{\cY}},\cP\rshad$ will be useful in the future study of the mechanism $\smash{\vec{\cY}}=\smash{\vec{\cY}}(\gamma,\smash{\vec{\cV}},\cP)$ of their observed $\dd_\cP$-\/triviality.
}
But the universal vector field $\smash{\vec{\cX}}(\gamma,\smash{\vec{\cV}},\cP)\in\ker\dd_{\cP}$ is identically zero
on~$\mathbb{R}^4$. Indeed, the Euler field $\smash{\vec{\cE}}=\smash{\vec{\cV}}$ is linear, yet it is readily seen
from the figures in~\cite{f16} that in every orgraph from the $1$-\/vector $\Or(\gamma)(\smash{\vec{\cV}\otimes\cP^{\otimes^{n-1}}})$,
the vertex with $\smash{\vec{\cV}}$ is differentiated at least twice (and at most thrice), so $\smash{\vec{\cX}}\equiv0$.
\end{example}

\begin{prop}
The flow $\smash{\dot{\cP}}=\Or(\text{tetrahedron $\boldsymbol{\gamma}_3$})(\cP)$ preserves the Nambu class of Poisson brackets,
$\{f,g\}_{\cP}=\varrho(x,y,z)\cdot\det\bigl(\dd(a,f,g)/\dd(x,y,z)\bigr)$ with arbitrary $\varrho$ and global Casimir $a$ on $\mathbb{R}^3$\textup{:}
the flow forces the nonlinear evolution $\dot{a}$, $\dot{\varrho}
$ with differential\/-\/polynomial r.-h.s.\\
$\bullet$\quad This flow $\dot{\cP}=\cQ(\cP)$ is \emph{not} Poisson\/-\/exact in terms of any vector field $\smash{\vec{\cY}}$ with 
differential\/-\/polynomial coefficients \textup{(}cubic in both $a$ and $\varrho$, of total differential order eight\textup{)}.\\
\mbox{ }
\quad The cocycle equation at hand, $\mathcal{E}(\boldsymbol{\gamma}_3,a,\varrho)=\{\dot{\cP}=\lshad\smash{\vec{\cY}},\cP\rshad\}$,
is a first\/-\/order PDE with differential\/-\/polynomial coefficients \textup{(}their skew\/-\/symmetry under permutations of $x$, $y$, $z$ is inherited from the property of the Jacobian determinant and from the transformation law for the density $\varrho$ in $\cP$\textup{)}. Whether this equation $\mathcal{E}$ does not admit any non\/-\/polynomial solutions $\smash{\vec{\cY}}(a_{|\sigma|\leqslant3}, \varrho_{|\tau|\leqslant2})$ 
is an open problem.
\end{prop}

{\small
\smallskip
\noindent\textbf{Acknowledgements.}
A.V.K.\ thanks the organizers of international workshop SQS'19 (August 26\/--\/31, 2019 in Yerevan, Armenia) for a warm atmosphere during the event. 
A~part of this research was done 
at the IH\'ES (Bures\/-\/sur\/-\/Yvette, France);
the authors are grateful to the \textsc{ratp} and \textsf{stif} for setting up stimulating working conditions.
The authors thank G.\,H.\,E.\,Duchamp and M.~Kontsevich for helpful discussions.

}

\end{document}